\documentclass[a4paper,reqno]{amsart}

\usepackage{amsmath}
\usepackage{paralist}
\usepackage{graphics}
\usepackage{epsfig}
\usepackage{amsfonts}
\usepackage{amssymb}
\usepackage{hyperref,cite}

\UseRawInputEncoding

\usepackage{amsthm}
\usepackage{enumerate}
\usepackage[mathscr]{eucal}
\usepackage{eqlist}

\newtheorem{theorem}{Theorem}[section]

\def\ifl{\iffalse }

\numberwithin{equation}{section}

\numberwithin{equation}{section}

\theoremstyle{remark}

\begin{document}
\title[The extremal problem for weighted combined energy and $\rho-$Nitsche type inequality]
{The extremal problem for weighted combined energy and $\rho-$Nitsche type inequality}

\author{Ting Peng}
\address{College of Mathematics and Information, China West Normal University,
       Nanchong 637009,P.R.CHINA}\email{2939358133@qq.com}

\author{Chaochuan Wang}
\address{College of Mathematics and Information, China West Normal University,
       Nanchong 637009,P.R.CHINA}\email{2904785435@qq.com}

\author{Xiaogao Feng$^{\ast}$}
\address{College of Mathematics and Information, China West Normal University,
       Nanchong 637009,P.R.CHINA}\email{fengxiaogao603@163.com}

\thanks{Research supported by the National Natural Science Foundation of China (Grant Nos.11701459 and 12271218).}
\thanks{$^{\ast}$ Corresponding author.}




\begin{abstract}
Let $A_1$ and $A_2$ be two circular annuli and let $\rho$ be a radial metric defined in the annuli $A_2$. We study the existence and uniqueness of  the extremal problem for weighted combined energy between $A_1$ and $A_2$, and obtain that the extremal mapping is a certain radial mapping. In fact, this extremal mapping generalizes the $\rho-$harmonic mapping and satisfies equation (2.7) obtained by mean of variation for weighted combined energy. Meanwhile, we get a $\rho-$Nitsche type inequality. This extends the results of Kalaj (J. Differential Equations, 268(2020)) and YTF (Arch. Math., 122(2024)), where they considered the case $\rho=1$ and $\rho=\frac{1}{|h|^{2}}$, respectively.

 Moreover, in the course of proving the extremal problem for weighted combined energy we also investigate the extremal problem for the weighted combined distortion (see Theorem 4.1). This extends the result obtained by Kalaj (J. London Math. Soc., 93(2016)).
\end{abstract}

\maketitle

\section{Introduction}
\quad

The extremal problem for harmonic energy is important in geometric function theory. Recently, this problem has been extended to setting for combined energy and weighted combined energy(\cite{FTP}, \cite{Ka1},\cite{YTF}). In this note, we mainly discuss the extremal problems for weighted combined energy.

Firstly, we give some basic definitions and notations. For two positive and distinct constants $r$ and $R$, let
\[{A_1=\left\{ z:1\leqslant \left| z \right|\leqslant r \right\}\quad \text{and} \quad A_2=\left\{ \omega :1\leqslant \left| \omega \right|\leqslant \small{R} \right\}}\]
be annuli in the complex plane $\mathbb{C}$.
We consider the class $\mathfrak{H}({A}_{1},{A}_{2})$ of all orientation preserving homeomorphisms
$h$ from ${A}_{1}$ onto ${A}_{2}$ keeping orders of the boundaries, namely,
\[|h(z)|=1 \quad \text{for} \quad |z|=1  \quad \text{and} \quad |h(z)|=R \quad \text{for} \quad |z|=r. \eqno{(1.1)}\]

For dealing with mapping of annuli, it is best suitable to make use of polar coordinates. Therefor we consider $z=x+iy=te^{i\theta}$, and the normal and tangential derivatives of $h$ are defined by
\[h_{N}=h_{t} \quad \text{and} \quad h_{T}=\frac{1}{t}h_{\theta}. \eqno{(1.2)}\]
The derivatives $h_{z}$ and $h_{\overline{z}}$ are expressed as
\[h_{z}=\frac{e^{-i\theta}} {2}(h_{N}-ih_{T}) \quad \text{and} \quad h_{\overline{z}}=\frac{e^{i\theta}} {2}(h_{N}+ih_{T}). \eqno{(1.3)}\]
So we obtain
\[|Dh|^{2}=2(|h_{z}|^{2}+|h_{\overline{z}}|^{2})=|h_{N}|^{2}+|h_{T}|^{2}, \eqno{(1.4)} \]
and the Jocobian determinant of $h$
\[J(z,h)=|h_{z}|^{2}-|h_{\overline{z}}|^{2}=\Im(\overline{h_{N}}h_{T}).\eqno{(1.5)}\]

The classical Gr\"{o}tzsch problem asks one to identify the affine mapping as the homeomorphism of least maximal distortion between two rectangles(see \cite{Gr}).
In 2005, Astala, Iwaniec, Martin and Oninnen \cite{AIMO} generalized the Gr\"{o}tzsch problem to the setting of mapping of finite distortion. Meanwhile, they discussed the extremal harmonic energy between two rectangles. Subsequently,
Iwaniec and Onninen \cite{IO1} considered the harmonic energy between annuli during discussing the total ennergy.  In 2010, Astala, Iwaniec and Martin \cite{AIM} studied the extremal mappings with smallest mean distortion between annuli when $\rho =1$ and $\rho(|w|) =\small{\frac{1}{\left| \omega \right|}}$. Subsequently,  Kalaj \cite{Ka4} generalized the result \cite{AIM} to the general metric $\rho(|\omega|)$. Precisely, he studied the following extremal problem
 \[\mathrm{\inf_{f\in\mathfrak{H}(A_{2},A_{1})} }\iint_{A_2}{\Bbb K}\left(\omega,f \right) \rho ^2\left( |\omega| \right) dudv, \eqno{(1.6)}\]
 where $\omega=u+iv$  and
\[{\Bbb K}\left(\omega,f \right)=\frac{2\left(|f_{\omega}|^{2}+|f_{\overline{\omega}}|^{2}\right)}{J(\omega,f)}.\]
For a related approach to the extremal problem for distortion, we refer to the papers (\cite{FTWS},\cite{MM}).

In 2020, Kalaj \cite{Ka1} researched the extremal problem for combined energy
\[\mathrm{\inf_{h\in\mathfrak{H}(A_{1},A_{2})} }\iint_{A_1}{\left( a^2\left| h_N \right|^2+b^2\left| h_T \right|^2 \right) dxdy}. \eqno{(1.7)}\]
Recently, on the basis of the result in \cite{Ka1}, Yang, Tang and Feng \cite{YTF} extended (1.7) to the following case
\[\mathrm{\inf_{h\in\mathfrak{H}(A_{1},A_{2})} }\iint_{A_1}{\left( a^2\left| h_N \right|^2+b^2\left| h_T \right|^2 \right)}\frac{1}{\left| h\left( z \right) \right|^{4}}dxdy. \eqno{(1.8)}\]
Very recently, Feng, Tang and Peng \cite{FTP} obtained the following situation
\[\mathrm{\inf_{h\in\mathfrak{H}(A_{1},A_{2})} }\iint_{A_1}{\left( a^2\left| h_N \right|^2+b^2\left| h_T \right|^2 \right)}\frac{1}{\left| h\left( z \right) \right|^{2\lambda}}dxdy. \eqno{(1.9)}\]

In this note, we shall generalize the extremal problem (1.7-1.9) to the more general case and mainly study the extremal problem for weighted combined energy
\[{E\left[ h \right] =\iint_{A_1}{\left( a^2\left| h_N \right|^2+b^2\left| h_T \right|^2 \right)}\rho ^2\left( |h| \right) dxdy},\eqno{(1.10)}\] where $a>0$, $b>0$.

It should be mentioned that all those problem have their root to famous Nitsche conjecture. To be precise,
Nitsche \cite{Ni1} conjectured that a necessary and sufficient condition for
existence of a harmonic homeomorphism between two annuli $A_1$ and $A_2$ is the following inequality
\[r\leq R+\sqrt{R^{2}-1} \quad  \text{or}  \quad  R\geq\frac{1}{2}(r+\frac{1}{r}). \eqno{(1.11)}\]

After various lower bounds for $R$ (see \cite{Ly}, \cite{We}, \cite{Ni2}), the Nitsche conjecture was finally solved by Iwaniec, Kovalev and
Onninen \cite{IKO} in 2011.

Recall that, given a Riemannian metric $\rho$ on ${A}_{2}$, a $C^{2}$ homeomorphism $h$ is said to be harmonic with
respect to $\rho$ (or $\rho-$harmonic) if
\[h_{z\overline{z}}(z)+\left(\log \rho^{2}\right)_{\omega}\circ hh_{z}h_{\overline{z}}=0. \eqno{(1.12)}\]

In view of Nitsche conjecture, Kalaj \cite{Ka4} proposed the so-called $\rho-$Nitsche conjecture as follows.

\textbf{$\rho-$Nitsche conjecture}   \emph{If there exists a $\rho-$harmonic homeomorphism of the annuli  $A_1$  onto $A_2$, then}
\[r\leq \exp\left(\int_{1}^{R}\frac{\rho(s)}{\sqrt{s^{2}\rho^{2}(s)+\beta_{0}}}ds\right), \eqno{(1.13)}\]
\emph{where}
\[\mathrm{\beta_{0}=-\inf_{1\leq s\leq R}(\rho^{2}(s))s^{2}}.\]

 In \cite{Ka3}, a partial result has been proved on $\rho-$Nitsche conjecture.  As $\rho =1$, the $\rho$-Nitsche conjecture happens to be the classical Nitsche conjecture. And when $ \rho(|\omega|)  =\left| \omega \right|^{-2}$ the inequality (1.13) just right becomes (1.11).
On the basis of this property, Feng and Tang \cite{FT} gave a positive answer to the $\rho$-Nitsche conjecture when $\rho(|\omega|)=|\omega|^{-2}$.

Under condition (1.13), Kalaj gave the answer to the extremal problem (1.6) and proved the following result.

\textbf{Theorem A} (see Lemma 4.2 in \cite{Ka4}) Let $f : A_2\rightarrow A_1$ be a homomorphism  between $A_2$ and $A_1$ there holds condition (1.13), the infimum of (1.6) is attained at the radial mapping.

Recently, Kalaj \cite{Ka1} considered the extremal problem (1.7) and gotten a
Nitsche type inequality
\[R\geq\cosh(\frac{b}{a}\log r)=\frac{1+r^{\frac{2b}{a}}}{2r^{\frac{b}{a}}}.\eqno{(1.14)}\]
Notice that when $a=b$, (1.14) turns to be the classic Nitsche inequality (1.11).

Under condition (1.14), Kalaj \cite{Ka1} gave the answer to the extremal problem (1.7) and got the following theorem.

\textbf{Theorem B} (see Theorem 3.1 in \cite{Ka1}) Under condition (1.14), the infimum of the combined energy
(1.7) is attained at the radial mapping $$h_{0*}\left( te^{i\theta} \right) =\frac{\left( 1-\mu +\left( 1+\mu \right) t^{\frac{2}{c}} \right)}{2t^{\frac{1}{c}}}e^{i\theta}. \eqno{(1.15)}\ $$
The minimizer is unique up to a rotation of annuli.

Subsequently, Yang, Tang, and Feng \cite{YTF} considered the extremal problem (1.8) and obtained $\frac{1}{\left| \omega \right|^2}$ -Nitsche type inequality
 \[{r\leqslant \left( R+\sqrt{R^2-1} \right) ^{\frac{a}{b}}}. \eqno{(1.16)}\]

Under condition (1.16), they \cite{YTF} gave the answer to the extremal problem (1.8) and got the following result.

\textbf{Theorem C} (see theorem 3.1 in \cite{YTF}) Under condition (1.16),  the infimum of (1.8)
is attained at the radial mapping
\[h_{2\ast}(te^{i\theta})=\frac{2(1+\sqrt{1+\frac{\delta}{b^{2}}})|z|^{\frac{b}{a}}}{(1+\sqrt{1+\frac{\delta}{b^{2}}})^{2}-|z|^{\frac{2b}{a}}\frac{\delta}{b^{2}}}
e^{i\theta},\eqno{(1.17)}\]
where $\delta$ satisfies
\[R=\frac{2r^{\frac{b}{a}}(1+\sqrt{1+\frac{\delta}{b^{2}}})}{(1+\sqrt{1+\frac{\delta}{b^{2}}})^{2}-r^{\frac{2b}{a}}\frac{\delta}{b^{2}}}. \eqno{(1.18)}\]
The minimizer is unique up to a rotation of annuli.

Very Recently, Feng, Tang and Peng \cite{FTP} investigated the extremal problem (1.9) and obtained
$\frac{1}{\left| \omega \right|^{\lambda}}$-Nitsche type inequality
 \[{r\leqslant {\left( R^{\lambda -1}+\sqrt{R^{2\left( \lambda -1 \right)}-1} \right) ^{\frac{1}{\lambda -1}\frac{a}{b}}}}, \eqno{(1.19)}\]
 where $\lambda \ne 1$.

Under condition (1.19), they \cite{FTP} gave the answer to the extremal problem (1.9) and got the following result.

\textbf{Theorem D} (see theorem 1.1 in \cite{FTP}) When $\lambda \ne 1$, under condition (1.19) and among all mappings in $\mathfrak{H}({A}_{1},{A}_{2})$ , the infimum of (1.9) is attained at radial mapping
$$
h_{\lambda\ast}\left( z \right) =\frac{2^{\frac{\mathbf{1}}{\lambda -1}}\left( 1+\sqrt{1+\frac{\epsilon}{b^2}} \right) ^{\frac{1}{\lambda -1}}\left| z \right|^{\frac{b}{a}}}{\left[ \left( 1+\sqrt{1+\frac{\varepsilon}{b^2}} \right) ^2-\left| z \right|^{\frac{2\left( \lambda -1 \right) b}{a}}\frac{\epsilon}{b^2} \right] ^{\frac{1}{\lambda -1}}}e^{i\theta}, \eqno{(1.20)}\
$$
where $\varepsilon$ satisfies $$
R=\frac{2^{\frac{\mathbf{1}}{\lambda -1}}\left( 1+\sqrt{1+\frac{\varepsilon}{b^2}} \right) ^{\frac{1}{\lambda -1}}r^{\frac{b}{a}}}{\left[ \left( 1+\sqrt{1+\frac{\varepsilon}{b^2}} \right) ^2-r^{\frac{2\left( \lambda -1 \right) b}{a}}\frac{\varepsilon}{b^2} \right] ^{\frac{1}{\lambda -1}}}. \eqno{(1.21)}\
$$
The minimizer is unique up to a rotation of annuli.

 In this paper, we shall generalize the extremal problems (1.7-1.9) to (1.10) and obtain the following main theorem.
\begin{theorem}
Under condition $$ r\leqslant \exp \left( \int_1^R{\frac{a\rho \left( s \right)}{\sqrt{b^{2}\rho ^2\left( s \right) s^2+\alpha _0}}ds} \right), \eqno{(1.22)}\ $$
 where
 \[\mathrm{\alpha _0=-\inf_{1\leqslant s\leqslant R}\rho ^2\left( s \right) s^2}.\]
 Then among all mapping $\mathfrak{H}({A}_{1},{A}_{2})$ , the infimum of (1.10) is attained at the radial mapping $ h_*\left( te^{i\theta} \right) =q^{-1}\left( t \right) e^{i\theta}, $ where
 $$q\left( s \right) =\exp \left( \int_1^s{\frac{a\rho \left( s \right)}{\sqrt{b^2s^2\rho ^2\left( s \right) +\alpha}}ds} \right), \eqno{(1.23)}\ $$ \\ and $\alpha$ satisfies $$ r =\exp \left( \int_1^R{\frac{a\rho \left( s \right)}{\sqrt{b^2s^2\rho ^2\left( s \right) +\alpha}}ds} \right). \eqno{(1.24)}\ $$
That is to say
\[E[h]\geqslant E[h_{\ast}].\]
The minimizer is unique up to a rotation of annuli.

\end{theorem}

\textbf{Remark 1.1}  \quad As $\rho=1$,  Theorem 1.1. happens to be Theorem B.

\textbf{Remark 1.2}  \quad When $\rho \left( |\omega| \right) =\frac{1}{\left| \omega\right|^2}$,  Theorem 1.1. turns to be Theorem C (see Theorem 3.1 in \cite{YTF}).

\textbf{Remark 1.3}  \quad If $\rho \left( |\omega| \right) =\frac{1}{\left| \omega \right|^{\lambda}}$, then Theorem 1.1. becomes Theorem D.\\

We end this section with the organization of the paper. In section 2, we will study the variation for weighted combined energy.
In section 3, we shall investigate the weighted combined energy of the radial mapping. In last section, we will consider the extremal problem for weighted combined distortion and use the relationship between weighted combined energy and weighted combined distortion to prove Theorem 1.1.

\section{variation for weighted combined energy}

Let $z=x+iy=te^{i\theta}\quad \text{and}\quad dxdy=\frac{i}{2}dzd\bar{z}$, and
by (1.3), we obtain
\begin{align}
h_N=h_ze^{i\theta}+h_{\bar{z}}e^{-i\theta},\  h_T=ih_ze^{i\theta}-ih_{\bar{z}}e^{-i\theta}, \tag{2.1}
\end{align}
and
\[
\left| h_N \right|^2=\left| h_z \right|^2+\left| h_{\bar{z}} \right|^2+2\mathrm{Re}h_z\bar{h}_z\frac{z}{\bar{z}}, \eqno{(2.2)}\]

\[\left| h_T \right|^2=\left| h_z \right|^2+\left| h_{\bar{z}} \right|^2-2\mathrm{Re}h_z\bar{h}_z\frac{z}{\bar{z}}.  \eqno{(2.3)} \]
We put (2.2) and (2.3) into (1.10) and get
\begin{align*}
 E\left[ h \right] &=\iint_{A_1}{\left( a^2\left| h_N \right|^2+b^2\left| h_T \right|^2 \right) \rho ^2\left( h \right) dxdy} \nonumber \\
&=\iint_{A_1}{\left\{ \left( a^2+b^2 \right) \left( \left| h_z \right|^2+\left| h_{\bar{z}} \right|^2 \right) +\left( a^2-b^2 \right) 2\mathrm{Re}h_zh_{\bar{z}}\frac{z}{\bar{z}} \right\} \rho ^2\left( h \right)}\frac{i}{2}dzd\bar{z}.
\end{align*}
We will try to look for minima of $E\left[ h \right]$ . If $h$ is such a minimum which is continuous, then, a variation $h_t$ of $h$ can be represented as $$ h+t\varphi \,\,, \varphi \in C^0\cap W_{0}^{1,2}\left( A_{1,}A_2 \right). $$
If $h$ is to a minimum, we must have $$ \frac{d}{dt}E\left[ h+t\varphi \right] =0, \ $$
That is to say
\begin{align*}
0={} & \frac{d}{dt}\left( \iint_{A_1}{\left\{ \left( a^2+b^2 \right) \rho ^2\left( h+t\varphi \right) \left[ \left( h+t\varphi \right) _z\left( \bar{h}+t\bar{\varphi} \right) _{\bar{z}}+\left( h+t\varphi \right) _{\bar{z}}\left( \bar{h}+t\bar{\varphi} \right) _z \right] \right.} \right.\\
     & \left. \left. +\left( a^2-b^2 \right) 2\mathrm{Re}\left( h+t\varphi \right) _z\left( \bar{h}+t\bar{\varphi} \right) _z\frac{z}{\bar{z}}\rho ^2\left( h+t\varphi \right) \right\} \frac{i}{2}dzd\bar{z} \right) \Big| _{t=0} \\
 ={} & \iint_{A_1}{\left\{ \left( a^2+b^2 \right) \left[ \rho ^2\left( h \right) \left( h_z\bar{\varphi}_{\bar{z}}+\bar{h}_{\bar{z}}\varphi _z+\bar{h}_z\varphi _{\bar{z}}+h_{\bar{z}}\bar{\varphi}_z \right) \right.\right.}\\
     & \left. +2\rho \left( \rho _h\varphi +\rho _{\bar{h}}\bar{\varphi} \right) \left( h_z\bar{h}_{\bar{z}}+\bar{h}_zh_{\bar{z}} \right) \right] \\
     & \left. +\left( a^2-b^2 \right) 2\mathrm{Re}\left[ \rho ^2\left( h \right) \left( h_z\bar{\varphi}_z+\bar{h}_z\varphi _z \right) +2\rho \left( \rho _h\varphi +\rho _{\bar{h}}\bar{\varphi} \right) h_z\bar{h}_z \right] \frac{z}{\bar{z}} \right\} idzd\bar{z}
\end{align*}
Let $ \varphi =\frac{\psi}{\rho ^2\left( h \right)} $, this becomes
\begin{align}
0={} & \iint_{A_1}{\left( a^2+b^2 \right)}\left\{ h_z\left( \bar{\psi}_{\bar{z}}-\frac{2\bar{\psi}}{\rho}\left( \rho _hh_{\bar{z}}+\rho _{\bar{h}}\bar{h}_{\bar{z}} \right) \right)\right. \nonumber \\
     & +\bar{h}_{\bar{z}}\left( \psi _z-\frac{2\psi}{\rho}\left( \rho _hh_z+\rho _{\bar{h}}\bar{h}_z \right) \right) \nonumber\\
     & +\bar{h}_z\left( \psi _{\bar{z}}-\frac{2\psi}{\rho}\left( \rho _hh_{\bar{z}}+\rho _{\bar{h}}\bar{h}_{\bar{z}} \right) \right)\nonumber\\
     & +h_{\bar{z}}\left( \bar{\psi}_z-\frac{2\bar{\psi}}{\rho}\left( \rho _hh_z+\rho _{\bar{h}}\bar{h}_z \right) \right)\nonumber\\
     & \left. +\frac{2}{\rho}\left( \rho _h\psi +\rho _{\bar{h}}\bar{\psi} \right) \left( h_z\bar{h}_{\bar{z}}+\bar{h}_zh_{\bar{z}} \right) \right\}\nonumber\\
     &  +\left( a^2-b^2 \right) 2\mathrm{Re}\frac{z}{\bar{z}}\left\{ h_z\left( \bar{\psi}_z-\frac{2\bar{\psi}}{\rho}\left( \rho _hh_z+\rho _{\bar{h}}\bar{h}_z \right) \right)\right.\nonumber\\
     & +\bar{h}_z\left( \psi _z-\frac{2\psi}{\rho}\left( \rho _hh_z+\rho _{\bar{h}}\bar{h}_{\boldsymbol{z}} \right) \right)\nonumber\\
     & \left. +\frac{2}{\rho}\left( \rho _h\psi +\rho _{\bar{h}}\bar{\psi} \right) h_zh_{\bar{z}} \right\}idzd\bar{z}\nonumber\\
 ={} & \iint_{A_1}{\left( a^2+b^2 \right)}2\mathrm{Re}\left( h_z\bar{\psi}_{\bar{z}}-\frac{2\rho _h}{\rho}h_zh_{\bar{z}}\bar{\psi} \right) idzd\bar{z}\nonumber\\
     & +\iint_{A_1}{\left( a^2+b^2 \right)}2\mathrm{Re}\left( \bar{h}_z\psi _{\bar{z}}-\frac{2\rho _{\bar{h}}}{\rho}\bar{h}_z\bar{h}_{\bar{z}}\psi \right) idzd\bar{z}\nonumber\\
     & +\iint_{A_1}{\left( a^2-b^2 \right)} 2\mathrm{Re}\left( h_z\bar{\psi}_z-\frac{2\rho _h}{\rho}h_{z}^{2}\bar{\psi} \right) \frac{z}{\bar{z}}idzd\bar{z}\nonumber\\
     & +\iint_{A_1}{\left( a^2-b^2 \right)}2\mathrm{Re}\left( \bar{h}_z\psi _z-\frac{2\rho _{\bar{h}}}{\rho}\bar{h}_{z}^{2}\psi \right) \frac{z}{\bar{z}}  idzd\bar{z}.\tag{2.4}
\end{align}
It is easy to get $$ 2\mathrm{Re}\left( \bar{h}_z\psi _z-\frac{2\rho _{\bar{h}}}{\rho}\bar{h}_{z}^{2}\psi \right) \frac{z}{\bar{z}}=2\mathrm{Re}\left( h_{\bar{z}}\bar{\psi}_{\bar{z}}-\frac{2\rho _h}{\rho}h_{\bar{z}}^{2}\bar{\psi} \right) \frac{\bar{z}}{z}.\eqno{(2.5)}\ $$
If $ h\in C^2 ,$ put (2.5) into (2.4), we can integrate by parts in (2.4) to get
\begin{align}
0={} & 2\mathrm{Re}\iint_{A_1}{\left( a^2+b^2 \right)}\left( h_{z\bar{z}}+\frac{2\rho _h}{\rho}h_zh_{\bar{z}} \right) \bar{\psi}idzd\bar{z}\nonumber\\
     & +2\mathrm{Re}\iint_{A_1}{\left( a^2+b^2 \right)}\left( h_{z\bar{z}}+\frac{2\rho _{\bar{h}}}{\rho}\bar{h}_z\bar{h}_{\bar{z}} \right) \psi idzd\bar{z}\nonumber\\
     & +2\mathrm{Re}\iint_{A_1}{\left( a^2-b^2 \right)}\left( h_{zz}\frac{z}{\bar{z}}+h_z\frac{1}{\bar{z}}+\frac{2\rho _h}{\rho}h_{z}^{2}\frac{z}{\bar{z}} \right) \bar{\psi}idzd\bar{z}\nonumber\\
     & +2\mathrm{Re}\iint_{A_1}{\left( a^2-b^2 \right)}\left( h_{\bar{z}\bar{z}}\frac{\bar{z}}{z}+h_{\bar{z}}\frac{1}{z}+\frac{2\rho _h}{\rho}h_{\bar{z}}^{2}\frac{\bar{z}}{z} \right) \bar{\psi}idzd\bar{z}\nonumber
\end{align}
\begin{align}
 ={} & 4\mathrm{Re}\iint_{A_1}{\left[ \left( h_{z\bar{z}}+\frac{2\rho _h}{\rho}h_zh_{\bar{z}} \right) \right.}\nonumber\\
     & \left.{+\frac{1}{2}\left( a^2-b^2 \right) \left( h_{zz}\frac{z}{\bar{z}}+h_z\frac{1}{\bar{z}}+h_{\bar{z}\bar{z}}\frac{\bar{z}}{z}+h_{\bar{z}}\frac{1}{z}+\frac{2\rho _h}{\rho}h_{z}^{2}\frac{z}{\bar{z}}+\frac{2\rho _h}{\rho}h_{\bar{z}}^{2}\frac{\bar{z}}{z} \right)} \right] \bar{\psi}idzd\bar{z}.\tag{2.6}
\end{align}
i.e.
\begin{align}
0=& \left( a^2+b^2 \right) \left( h_{z\bar{z}}+\frac{2\rho _h}{\rho}h_zh_{\bar{z}} \right) \nonumber\\
&+\frac{1}{2}\left( a^2-b^2 \right) \left( h_{zz}\frac{z}{\bar{z}}+h_z\frac{1}{\bar{z}}+h_{\bar{z}\bar{z}}\frac{\bar{z}}{z}+h_{\bar{z}}\frac{1}{z}+\frac{2\rho _h}{\rho}h_{z}^{2}\frac{z}{\bar{z}}+\frac{2\rho _h}{\rho}h_{\bar{z}}^{2}\frac{\bar{z}}{z} \right). \tag{2.7}
\end{align}
By (1.3), we then compute
\[h_{z\bar{z}}=\frac{1}{4t^2}\left( t^2h_{tt}+th_t+h_{\theta \theta} \right),\eqno{(2.8)}\]
\[h_zh_{\bar{z}}=\frac{1}{4t^2}\left( t^2h_{t}^{2}+h_{\theta}^{2} \right),\eqno{(2.9)}\]
\[h_{zz}=\frac{1}{4t^2}\left( t^2h_{tt}+2ih_{\theta}-2ith_{t\theta}-th_t-h_{\theta \theta} \right) \frac{\bar{z}}{z},\eqno{(2.10)}\]
\[h_{\bar{z}\bar{z}}=\frac{1}{4t^2}\left( t^2h_{tt}-2ih_{\theta}+2ith_{t\theta}-th_t-h_{\theta \theta} \right) \frac{z}{\bar{z}},\eqno{(2.11)}\]
and
\[h_{z}^{2}=\frac{1}{4t^2}\left( t^2h_{t}^{2}-2ith_th_{\theta}-h_{\theta}^{2} \right) \frac{\bar{z}}{z},\eqno{(2.12)}\]
\[h_{\bar{z}}^{2}=\frac{1}{4t^2}\left( t^2h_{t}^{2}+2ith_th_{\theta}-h_{\theta}^{2} \right) \frac{z}{\bar{z}}.\eqno{(2.13)}\]
Inserting (2.8-2.13) into (2.7), we conclude
\begin{align}
2a^2t^2h_{tt}+2a^2th_t+2b^2h_{\theta \theta}+\frac{2\rho _h}{\rho}\left( 2a^2t^2h_{t}^{2}+2b^2h_{\theta}^{2} \right) =0.\tag{2.14}
\end{align}

\textbf{Remark 2.1}  If $a^{2}=b^{2}$ , (2.7) becomes equation (1.12).

\section{weighted combined energy of radial mapping}

For mappings of annuli, it is natural to examine the radial mappings.
Assume
\[h(z)=H(t)e^{i\theta} \quad  \text{for} \quad  z=te^{i\theta},\] where $H:[1,r]\rightarrow[1,R]$ is continuous strictly increasing and satisfies the boundary conditions (1.1).
Direct calculations yield$$
h_t=\dot{H}\left( t \right) e^{i\theta},\quad h_{tt}=\ddot{H}\left( t \right) e^{i\theta}, \eqno{(3.1)}\ $$
and
$$
h_{\theta}=iH\left( t \right) e^{i\theta},\quad h_{\theta \theta}=-H\left( t \right) e^{i\theta}. \eqno{(3.2)}\ $$
Inserting (3.1) and (3.2) into equation (2.14), because$$
\frac{2\rho _h}{\rho}=\frac{\rho ^{\prime}\left( H\left( t \right) \right) \bar{z}}{\rho \left( H\left( t \right) \right) \left| z \right|},
$$
we obtain that
$$ 2b^2H\left( t \right) -2a^2t\dot{H}\left( t \right) -2a^2t^2\ddot{H}\left( t \right) =\left( 2a^2t^2\dot{H}^2\left( t \right) -2b^2H^2\left( t \right) \right) \frac{\rho ^{\prime}\left( H\left( t \right) \right)}{\rho \left( H\left( t \right) \right)}. $$
Set $t=e^x$, $y=H\left( t \right)=H\left( e^x \right)$, then $$ y^{\prime}=\dot{H}\left( t \right) e^x\quad \text{and}\quad
y''=\ddot{H}\left( t \right) e^{2x}+\dot{H}\left( t \right) e^x.
$$
Thus
 $$ \dot{H}\left( t \right) =\frac{y^{\prime}}{e^x}\quad \text{and}\quad \ddot{H}\left( t \right) =\frac{y''-y^{\prime}}{e^{2x}}. \eqno{(3.3)}
$$
We put (3.3) into (3.2) and obtain $$ 2b^2y-2a^2y''=\left( 2a^2{y^{\prime}}^2-2b^2y^2 \right) \frac{\rho \prime\left( y \right)}{\rho \left( y \right)}. \eqno{(3.4)}
$$
Moreover set $y^{\prime}=z,$ then
\begin{align}
y''=\frac{dz}{dy}\frac{dy}{dx}=zz^{\prime}.\tag{3.5}
\end{align}
Put (3.5) into (3.4), we have
\begin{align}
2b^2y-2a^2zz^{\prime}=\left( 2a^2z^2-2b^2y^2 \right) \frac{\rho ^{\prime}\left( y \right)}{\rho \left( y \right)}.\tag{3.6}
\end{align}
Let
\[\omega=2a^{2}z^{2}-2b^{2}y^{2},\eqno{(3.7)} \] then
\begin{align}
\omega ^{\prime}=4a^2zz^{\prime}-4b^2y. \tag{3.8}
\end{align}
Hence, the equation (3.6) becomes
\begin{align}
\frac{\omega ^{\prime}}{\omega}=-2\frac{\rho ^{\prime}\left( y \right)}{\rho \left( y \right)}.\tag{3.9}
\end{align}
And the we get $\omega=\frac{\alpha}{\rho ^2\left( y \right)}. $ Together with (3.7), we have
\[\frac{\alpha}{\rho ^{2}\left( y \right)}=2a^{2}z^{2}-2b^{2}y^{2}, \eqno{(3.10)} \]
and
\[y^{\prime}=z=\frac{dy}{dx}=\sqrt{\frac{2b^{2}y^{2}+\frac{\alpha}{\rho ^{2}\left( y \right)}}{2a^{2}}},\eqno{(3.11)} \]
which can be written as
\[dx={\frac{a}{\sqrt{b^{2}y^{2}+\frac{\alpha}{\rho ^{2}\left( y \right)}}}}dy. \eqno{(3.12)} \]
Thus, for $1\leqslant y\leqslant R$, we have
\[x\left( y \right) =\int_1^y{\frac{a\rho \left( s \right)}{\sqrt{b^{2}s^{2}\rho ^{2}\left( s \right)+\alpha}}ds}. \eqno{(3.13)} \]
Furthermore,
\[t=e^{x}=\exp\left(\int_1^y{\frac{a\rho \left( s \right)}{\sqrt{b^{2}s^{2}\rho ^{2}\left( s \right)+\alpha}}ds}\right).\]
Therefore, we get the radial mapping
\[h\left( te^{i\theta}\right)=q^{-1}\left( t\right)e^{i\theta},\eqno{(3.14)} \]
where
\[q\left( s \right) =\exp \left( \int_1^s{\frac{a\rho \left( s \right)}{\sqrt{b^{2}s^{2}\rho ^{2}\left(s \right)+\alpha}}ds} \right),\eqno{(3.15)} \] $\alpha$ satisfies
\[ \alpha\geqslant -b^{2}s^{2}\rho^{2}\left( s \right)\quad \text{for}\quad 1\leqslant s\leqslant R, \]
and
\[\quad r=\exp \left( \int_1^R{\frac{a\rho \left( s \right)}{\sqrt{b^{2}s^{2}\rho^{2}(s)+\alpha}}ds} \right).\eqno{(3.16)} \]
If we take
\[\alpha_0=-\mathrm{\inf_{1\leqslant s\leqslant R}}b^{2}\rho ^{2}\left( s\right)s^{2} , \]
then we have the following $\rho-$Nitsche type inequality
\[r\leqslant \exp\left(\int_1^R{\frac{a\rho \left( s\right)}{\sqrt{b^{2}\rho^{2}\left( s\right)s^{2}+\alpha_0}}}ds\right).
\eqno{(3.17)} \]

\textbf{Remark 3.1}
If $\rho=1$ or $\rho=\frac{1}{\left| \omega\right|^2}$ the inequality (3.17) becomes (1.16). Moreover, when $\frac{a}{b}=1 $, it is converted to inequality (1.13). Additionally, when $\rho=\frac{1}{\left| \omega\right|^\lambda}$, the inequality (3.17) turns to (1.19).

\section{Proof of Theorem 1.1}

In this section, we firstly consider the extremal problem for weighted combined distortion (4.1). And in view of the relationship between (4.1) and (1.10) we give the proof Theorem 1.1.

\subsection{The minimizer of the weighted combined distortion. }
For $a, b>0$ and $f\in\mathscr{\mathfrak{H}} \left( A_2, A_1\right)$, we define the weighted combined distortion by
\[\mathcal{K} \left[ f \right] =\iint_{A_2}{\frac{b^2\left| \triangledown \varrho \right|^2+a^2\varrho ^2\left| \triangledown \Theta \right|}{J\left( \omega, f \right)}}\rho ^2\left( \left| \omega \right| \right)dudv, \eqno{(4.1)} \]
where $\omega=se^{i\tau}=u+iv \quad \text{and} \quad f\left( \omega\right)=\varrho e^{i\Theta}.$ \\

\begin{theorem}
Under the condition (1.22), for $f\in \mathfrak{H}({A}_{2},{A}_{1})$, the weighted combined distortion $\mathcal{K}[f]$ attains its minimum for a radial mapping $f^{\ast}(se^{i\tau})=q(s)e^{i\tau}$,
where $q(s)$ is defined as (3.15).
That is to say,
$$\mathcal{K}[f]\geq\mathcal{K}[f^*].$$
\end{theorem}

\begin{proof} Firstly, we consider the map $f^{\ast}(se^{i\tau})=q(s)e^{i\tau}$ and obtain
\begin{align}
\mathcal{K} \left[ f^* \right]
={}& \iint_{A_2}{\frac{b^2\left| \triangledown \varrho \right|^2+a^2\varrho ^2\left| \triangledown \Theta \right|}{J\left( \omega, f^* \right)}}\rho ^2\left( \left| \omega \right| \right)dudv \nonumber\\
={}& \iint_{A_2}{\frac{b^2q^{\prime} \left( s\right)+a^2q^2\left( s\right)\frac{1}{s^2}}{\frac{1}{s}q^{\prime}\left( s\right)q\left( s\right)}}dudv \nonumber\\
={}& 4\pi \int_1^R{\frac{ab^2s\rho ^3\left( s \right)}{\sqrt{b^2s^2\rho ^2\left( s \right) +\alpha}}}ds+2\pi \alpha \int_1^R{\frac{a\rho \left( s \right)}{\sqrt{b^2s^2\rho ^2\left( s \right) +\alpha}}}ds.\tag{4.2}
\end{align}

Next, we will prove that the minimizer of the weighted combined distortion is obtained at $f^{\ast}$.

By the following formulas (see Lemma 4.3 in \cite{Ka1})
\[\left| \triangledown \varrho \right|\geqslant \left| f\right|_N,\quad \left| \triangledown \Theta\right|\geqslant \Im \left[ \frac{f_T}{f}\right], \]
and for $a_*, b_*, A_*, B_* \in \mathbb R,$ we get the following general inequality
\begin{align}
       & {\frac{b^2\left| \triangledown \varrho \right|^2+a^2\varrho ^2\left| \triangledown \Theta \right|^2}{J\left( \omega, f \right)}}\rho ^2\left( \left| \omega \right| \right)\nonumber \\
 \geq{}& \frac{b^2{\left| h \right|_N}^2+a^2\varrho ^2\left| \Im \left[ \frac{f_T}{f} \right] \right|^2}{J\left( \omega ,f \right)}\rho ^2\left( s \right)\nonumber \\
 \geq{}& {\frac{(b^2-a^2a_{*}^{2})|f|_{N}^{2}+(a^2-b^2b_{*}^{2})\varrho ^2\left| \Im \left[ \frac{f_T}{f} \right] \right|^2+2aa_{*}bb_{*} J( \omega, f)}{J( \omega, f)}}\rho ^2 \left( s\right).\tag{4.3}
\end{align}

We divide the proof into two cases

\textbf{Elastic case: $\alpha>0$.}
Let $a_*=\frac{b}{a},\quad b_*=\frac{as\rho(s)}{\sqrt{b^2s^2\rho^2(s)+\alpha}}\quad\text{and}\quad B_*=\frac{\sqrt{b^2s^2\rho^2(s)+\alpha}} {a|f|\rho(s)},$ then (4.3) becomes
\begin{align}
        &{\frac{b^2| \triangledown \varrho |^2+a^2\varrho ^2| \triangledown \Theta |^2}{J( \omega, f )}}\rho ^2(| \omega| )\nonumber \\
  \geq{}& \frac{(a^2-b^2b_{*}^{2})\varrho ^2\left| \Im \left[ \frac{f_T}{f} \right] \right|^2+2aa_{*}bb_{*}J( \omega, f)}{J( \omega, f)}\rho ^2(s) \nonumber\\
  \geq{}& (a^2-b^2b_{*}^{2})\rho ^2(s)[2B_*|f|\Im[\frac{f_T}{f}]-B_{*}^{2}J(\omega, f)]+2b^2b_{*}\rho^2(s)\nonumber \\
     ={}& \frac{2a\alpha\rho(s)}{\sqrt{b^2s^2\rho^2(s)+\alpha}}\Im[\frac{f_T}{f}]-\frac{\alpha}{|f|^2}J(\omega, f)+\frac{2ab^2s\rho^3(s)}{\sqrt{b^2s^2\rho^2(s)+\alpha}}.\tag{4.4}
\end{align}
Integrating both sides, we conclude
\begin{align}
    {}& \iint_{{A}_{2}}{\frac{b^2| \triangledown \varrho |^2+a^2\varrho ^2| \triangledown \Theta |^2}{J( \omega, f )}}\rho ^2(| \omega|)dudv\nonumber \\
\geq{}& \iint_{{A}_{2}}\frac{2a\alpha\rho(s)}{\sqrt{b^2s^2\rho^2(s)+\alpha}}\Im[\frac{f_T}{f}]dudv
   -\iint_{{A}_{2}}\frac{\alpha}{|f|^2}J(\omega,f)dudv\nonumber\\
   {}&+\iint_{{A}_{2}}\frac{2ab^2s\rho^3(s)}{\sqrt{b^2s^2\rho^2(s)+\alpha}}dudv\nonumber\\
   ={}& 4\pi \alpha\int_{1}^{R}\frac{a\rho(s)}{\sqrt{b^2s^2\rho^2(s)+\alpha}}ds-2\pi \alpha\int_{1}^{r}\frac{1}{\varrho}d\varrho
    +4\pi \int_1^R{\frac{ab^2s\rho ^3\left( s \right)}{\sqrt{b^2s^2\rho ^2\left( s \right) +\alpha}}}ds\nonumber\\
   ={}& 4\pi \int_1^R{\frac{ab^2s\rho ^3\left( s \right)}{\sqrt{b^2s^2\rho ^2\left( s \right) +\alpha}}}ds+2\pi \alpha\int_{1}^{R}\frac{a\rho(s)}{\sqrt{b^2s^2\rho^2(s)+\alpha}}ds\nonumber\\
      & + 2\pi \alpha\int_{1}^{R}\frac{a\rho(s)}{\sqrt{b^2s^2\rho^2(s)+\alpha}}ds-2\pi\alpha\ln r\nonumber\\
   ={}& \mathcal{K}[f^*]+X.\tag{4.5}
\end{align}
From (3.16), we can have
\[X=2\pi \alpha\int_{1}^{R}\frac{a\rho(s)}{\sqrt{b^2s^2\rho^2(s)+\alpha}}ds-2\pi\alpha\ln r=0.\]
Thus
$$\mathcal{K}[f]\geq\mathcal{K}[f^*].$$

\textbf{Non-elastic case:} $0>\alpha>\alpha_0=-\mathrm{\inf_{1\leqslant s\leqslant R}}b^{2}\rho ^{2}\left( s\right)s^{2}$. We take $b_*=\frac{a}{b},\quad a_*=\frac{\sqrt{b^2s^2\rho^2(s)+\alpha}}{as\rho(s)}\quad \text{and}\quad A_*=\frac{s}{|f|},$ then (4.3) changes into
\begin{align}
       & {\frac{b^2| \triangledown \varrho |^2+a^2\varrho ^2| \triangledown \Theta |^2}{J( \omega, f )}}\rho ^2(| \omega| ) \nonumber\\
 \geq{}& {\frac{(b^2-a_{*}^{2}a^2)|f|_{N}^{2}+2a^2a_*J(\omega, f)}{J(\omega, f)}}\rho^2(s)\nonumber\\
 \geq{}& (b^2-a_{*}^{2})\rho^2(s)[2A_*|f|_{N}-A_*J(\omega,f)]+2a^2a_*\rho^2(s)\nonumber\\
    ={}& -\frac{2\alpha |f|_s}{s|f|}+\frac{\alpha}{|f|^2}J(\omega, f)+\frac{2a\sqrt{b^2s^2\rho ^2\left( s \right) +\alpha}}{s}\rho(s).\tag{4.6}
\end{align}
Integrating both sides, we get
\begin{align}
      & \iint_{{A}_{2}}{\frac{b^2| \triangledown \varrho |^2+a^2\varrho ^2| \triangledown \Theta |^2}{J( \omega, f )}}\rho ^2(| \omega|)dudv \nonumber\\
\geq{}& -\iint_{{A}_{2}}\frac{2\alpha|f|_s}{s|f|}dudv+\iint_{{A}_{2}}\frac{\alpha}{|f|^2}J(\omega,f)dudv \nonumber\\
   {}& +\iint_{{A}_{2}}\frac{2a\sqrt{b^2s^2\rho ^2\left( s \right) +\alpha}}{s}\rho(s)dudv\nonumber\\
   ={}& -4\pi \alpha\int_{1}^{r}\frac{1}{\varrho}d\varrho+2\pi \alpha\int_{1}^{r}\frac{1}{\varrho}d\varrho+4\pi \int_1^R{\frac{ab^2s\rho ^3\left( s \right)}{\sqrt{b^2s^2\rho ^2\left( s \right) +\alpha}}}ds\nonumber\\
      & + 4\pi \alpha\int_{1}^{R}\frac{a\rho(s)}{\sqrt{b^2s^2\rho^2(s)+\alpha}}ds\nonumber\\
   ={}& 4\pi \int_1^R{\frac{ab^2s\rho ^3\left( s \right)}{\sqrt{b^2s^2\rho ^2\left( s \right) +\alpha}}}ds+2\pi \alpha\int_{1}^{R}\frac{a\rho(s)}{\sqrt{b^2s^2\rho^2(s)+\alpha}}ds\nonumber\\
   ={}& \mathcal{K}[f^*].\tag{4.7}
\end{align}

$\mathbf{Uniqueness}$

Assume
\[f(se^{i\tau})=\varrho(s,\tau)e^{i\Theta(s,\tau)}.\eqno{(4.8)} \]
We prove the uniqueness to elastic case. In (4.3) equality is achieved if and only if
\[| \triangledown \varrho |= | f|_{N},\eqno{(4.9)}\]
\[| \triangledown \Theta|=\Im [ \frac{f_T}{f}],\eqno{(4.10)}\]
\[a_*a|f|_{N}=b_*b\rho(s)\Im[\frac {f_{T}}{f}],\eqno{(4.11)}\]
\[|h|_{N}\rho(s)\Im[\frac{f_{T}}{f}]=J(\omega, f)=\Im(\overline{f_{_{N}}}f_{T}), \eqno{(4.12)}\]
\[\rho(s)\Im[\frac{f_T}{f}]=B_*J(\omega,f). \eqno{(4.13)}\]
By (4.9) and (4.10), we can obtain
\[\sqrt{(\frac{\varrho_{\tau}}{s})^2+\varrho_{s}^2}=\varrho_{s}\quad\text{and}\quad \sqrt{(\frac{\Theta_{\tau}}{s})^2+\Theta_{s}^2}=\frac{\Theta_{\tau}}{s},\]
then $$ \varrho_{\tau}=0,\quad\quad\Theta_{s}=0.\eqno{(4.14)} $$
By $a_*=\frac{b}{a},  b_*=\frac{as\rho(s)}{\sqrt{b^2s^2\rho^2(s)+\alpha}}, B_*=\frac{\sqrt{b^2s^2\rho^2(s)+\alpha}} {a|f|\rho(s)}$ and (4.11), we have
\[\varrho_{s}=\frac{a\rho(s)}{\sqrt{b^2s^2\rho^2(s)}}\varrho\Theta_{\tau}.\eqno{(4.15)}\]
By (4.13) and (4.14), we get
\[\varrho\Theta_{\tau}=\frac{\sqrt{b^2s^2\rho^2(s)+\alpha}}{{a\rho(s)}}\varrho_{s}\Theta_{\tau}.\eqno(4.16)\]
Put (4.15) into (4.16), we can obtain
\[\varrho\Theta_{\tau}=\frac{\sqrt{b^2s^2\rho^2(s)+\alpha}}{a\rho{(s)}}\varrho_{s}
=\frac{\sqrt{b^2s^2\rho^2(s)+\alpha}}{a\rho(s)}\varrho(s)\Theta_{\tau}. \]
And then
$$\Theta_{\tau}=1.\eqno{(4.17)} $$
Due to $\Theta_{s}=0,$ we have
$\Theta=\tau+\beta,$ where $\beta$ is a real constant. By (4.16) and (4.17) we infer that
\[\frac{\varrho_{s}}{\varrho}=\frac{a\rho(s)}{\sqrt{b^2s^2\rho^2(s)+\alpha}}=\frac{q^{\prime}(s)}{q(s)},\]
namely,
\[\varrho=q(s).\eqno{(4.18)} \]
Thus
\[f(se^{i\tau})=q(s)e^{i\tau}=f^{\ast}(se^{i\tau}).\]
The uniqueness to non-elastic case is obtained in a similar way.
We complete the proof of Theorem 4.1.
\end{proof}

\subsection{The relationship between weighted combined energy and weighted combined distortion.}

In this subsection, we will prove the following theorem.\\
\begin{theorem}
Assume that $h$ is a map of $A_1$ to $A_2$, and let $f=h^{-1}$. If
\[E[h]=\iint_{A_{1}}(a^2|h_{N}|^2+b^2|h_{T}|^2)\rho^2(|h|)dxdy, \eqno{(4.19)} \]
then $$ E[h]=\mathcal{K}[f],$$ where $\mathcal{K}[f]$ is defined in (4.1).
\end{theorem}
\begin{proof}
Assume
$$h(\varrho e^{i\Theta})=\omega=se^{i\tau}. $$
Taking the derivation of $s$ and $\tau$ on both sides respectively, we obtain
\[h_{\varrho}\varrho_{s}+h_{\Theta}\Theta_{s}=e^{i\tau},\]
\[\quad h_{\varrho}\varrho_{\tau}+h_{\Theta}\Theta_{\tau}=ise^{i\tau}.\]
Direct calculation
\[h_{N}=\frac{e^{i\tau}(\Theta_{\tau}-is\Theta_{s})}{\varrho_{s}\Theta_{\tau}-\varrho_{\tau}\Theta_{s}}, \eqno{(4.20)}\]
\[h_{T}=\frac{-e^{i\tau}(\varrho_{\tau}-is\varrho_{s})}{\varrho(\varrho_{s}\Theta_{\tau}-\varrho_{\tau}\Theta_{s})}. \eqno{(4.21)}\]
If $ f(\omega)=\varrho e^{i\Theta}$ , then $ J(\omega,f)=\frac{\varrho}{s}(\varrho_{s}\Theta_{\tau}-\varrho_{\tau}\Theta_{s}),$
thus we can obtain
\[\varrho_{s}\Theta_{\tau}-\varrho_{\tau}\Theta_{s}=\frac{s}{\varrho}J(\omega,f).\eqno{(4.22)}\]
Meanwhile
\[|\Theta_{\tau}-is\Theta_{s}|^{2}=s^{2}|\triangledown \Theta|^{2},\eqno{(4.23)}\]
\[|\varrho_{\tau}-is\varrho_{s}|^{2}=s^{2}|\triangledown \varrho|^{2}.\eqno{(2.24)}\]
Put (4.20-4.24) into (4.19), we get
\[\iint_{A_{1}}(a^2|h_{N}|^2+b^2|h_{T}|^2)\rho^2(|h|)dxdy=\iint_{A_2}{\frac{b^2\left| \triangledown \varrho \right|^2+a^2\varrho ^2\left| \triangledown \varTheta \right|}{J\left( \omega, f \right)}}\rho ^2\left( \left| \omega \right| \right)dudv.\]
That is to say
 $$ E[h]=\mathcal{K}[f].$$
We complete the proof of Theorem 4.1.
\end{proof}

{\bf Proof of Theorem 1.1}
Combining Theorem 4.1 and Theorem 4.2, we  complete the proof of Theorem 1.1.

{\bf Acknowledgements} \quad  The authors would like to thank to the referee for a very careful reading of manuscript.

\bibliographystyle{amsplain}

\end{document}